\documentclass[12pt]{article}

\usepackage[numbers]{natbib}
\usepackage[utf8]{inputenc}
\usepackage{amsmath,amsthm,amssymb,mathtools,enumerate,verbatim,url,natbib,amsbsy}
\usepackage{bbm,fancyhdr,amscd,latexsym,graphicx,pdfsync,blkarray,multirow,hyperref,framed}

\RequirePackage{kpfonts}
\RequirePackage[sf,bf,small,raggedright]{titlesec}
\usepackage[margin=0.7in]{geometry}
\theoremstyle{plain}
\newtheorem{theorem}{Theorem}[section]
\newtheorem{lemma}[theorem]{Lemma}
\newtheorem{proposition}[theorem]{Proposition}

\theoremstyle{definition}

\theoremstyle{remark}

\newcommand{\E}{\mathbf{E}}
\renewcommand{\P}{\mathbf{P}}
\newcommand{\p}[1]{\P\left\{{#1}\right\}}

\newcommand{\transpose}{^{\mathsf{T}}}
\newcommand{\tv}[2]{\operatorname{TV}\left(#1, #2\right)}
\newcommand{\TV}{\operatorname{TV}}
\newcommand{\kl}[2]{\operatorname{KL}\left(#1 \parallel #2\right)}
\newcommand{\tvn}[4]{\tv{\ncal(#1,#2)}{\ncal(#3,#4)}}
\newcommand{\R}{\mathbb{R}}
\newcommand{\ncal}{\mathcal{N}}
\newcommand{\tr}{\operatorname{tr}}
\newcommand{\range}{\operatorname{range}}
\newcommand{\rank}{\operatorname{rank}}
\def\red#1{#1}

\usepackage[dvipsnames]{xcolor}

\author{Luc Devroye\thanks{Supported by NSERC Grant A3456.}\\McGill University
	\and Abbas Mehrabian\thanks{Supported by an IVADO-Apog\'ee-CFREF Postdoctoral Fellowship. Email:
		\href{mailto:abbas.mehrabian@gmail.com}{abbas.mehrabian@gmail.com}.}\\McGill University
	\and Tommy Reddad\thanks{Supported by  NSERC PGS D Scholarship 396164433.}\\McGill University}

\title{The total variation distance between \\ high-dimensional Gaussians\\
with the same mean}

\begin{document}
	
	\maketitle
	
	\begin{abstract}
		Given two high-dimensional Gaussians with the same mean,
		we prove a lower and an upper bound for their total variation distance, which are within a constant factor of one another.\footnote{In an earlier version, tight bounds were claimed for the total variation distance between two general Gaussians. But the proof of the upper bound was incorrect, and we removed the flawed bound from the paper. Later, Arbas, Ashtiani, and Liaw (\cite[Theorem 1.8]{tight_bounds}) proved tight bounds for the total-variation distance between two general Gaussians, solving the original problem.}
	\end{abstract}
	
	\section{Introduction}
	The Gaussian (or normal) distribution is perhaps the most important distribution in probability theory due to the central limit theorem.
	For a positive integer $d$, a vector $\mu \in \R^d$, and a positive definite matrix $\Sigma$, the Gaussian distribution with mean $\mu$ and covariance matrix $\Sigma$ is a probability distribution over $\R^d$, denoted by $\ncal(\mu,\Sigma)$, with density 
	\[
	\det(2\pi \Sigma)^{-1/2} \exp \big(- (x-\mu)\transpose \Sigma^{-1} (x-\mu)/2\big) \qquad \forall x\in\R^d.
	\]
	We denote by $N(\mu,\Sigma)$ a random variable with this distribution. Note that if $X\sim \ncal(\mu,\Sigma)$ then $\E X = \mu$ and $\E X X\transpose = \Sigma$.
	
	If the covariance matrix is positive semi-definite but not positive definite, the Gaussian distribution 
	is singular on $\R^d$ but 
	has a density with respect to  a Lebesgue measure on an affine subspace: let
	$r$ be the rank of $\Sigma$,
	and let 
	$\range (\Sigma)$
	denote the range (also known as the image or the column space) of $\Sigma$.
	Let $\Pi$ be a $d\times r$ matrix whose columns form an orthonormal basis for $\range (\Sigma)$. Then the matrix $\widehat{\Sigma}\coloneqq \Pi\transpose \Sigma \Pi$ has full rank $r$, and $\ncal(\mu,\Sigma)$ has density given by
	\[
	\det(2\pi \widehat{\Sigma})^{-1/2} \exp \big(- (x-\mu)\transpose \Pi \widehat{\Sigma}^{-1} \Pi\transpose (x-\mu)/2\big)
	\]
	with respect to the $r$-dimensional Lebesgue measure on $\mu +\range (\Sigma)$. The density is zero outside this affine subspace. For general background on high-dimensional Gaussian distributions (also called multivariate normal distributions), see~\cite{Tong,Vershynin}.
	
	Given two Gaussian distributions, our goal is to understand how different they are.
	Our measure of similarity is the {\em total variation distance}, which, for any two distributions $P$ and $Q$ over $\R^d$, is defined as 
	\[ \tv{P}{Q} \coloneqq \sup_{A\subseteq \R^d} |P(A)-Q(A)|.\]
	If $P$ and $Q$ have densities $p$ and $q$, then it is easy to verify that the set $A\coloneqq \{x : p(x)>q(x)\}$ attains the supremum here, and this observation leads to the identity
	\begin{equation}\label{tvl1}
	\tv{P}{Q} = \frac12 \int_{\R^d} |p(x) - q(x)| \, dx, 
	\end{equation}
	that is, the total variation distance is half the $L^1$ distance.
	In the following, we will sometimes write $\tv{X}{Y}$ for $\tv{P}{Q}$, where $X$ and $Y$ are random variables distributed as $P$ and $Q$, respectively. Observe that $\tv{P}{Q}$ is a metric and is always between 0 and 1. For a survey on  measures of distance between distributions and  inequalities between them, see~\cite{Verdu}.
	
	We have seen that the total variation distance can be written as an integral or as a supremum,
	but in general there is no known closed form for it.
	In this note, given two Gaussians with the same mean, we give closed-form lower and upper bounds for their total variation distance,
	which are within a constant factor of one another.
	If the Gaussians have different means, we give only a lower bound and leave a tight characterization as an open problem.\footnote{This problem has been solved; see \cite[Theorem 1.8]{tight_bounds}.}
	
	\smallskip
	\noindent \textbf{Open Problem.}
	Find closed-form lower and upper bounds for the total variation distance between two high-dimensional Gaussians that are within a constant factor of one another.
	\smallskip
		
	Note that if 
	$\mu_1 + \range(\Sigma_1)
	\neq \mu_2 + \range(\Sigma_2)$,
	in particular if $\rank(\Sigma_1)\neq\rank(\Sigma_2)$,
	then we have $
	\tvn{\mu_1}{\Sigma_1}{\mu_2}{\Sigma_2}=1$, 
	since the intersection of the supports have zero Lebesgue measure.
	Another trivial case is when $\mu_1=\mu_2$ and $\Sigma_1=\Sigma_2$, in which case the total variation distance is zero.
	We will not explicitly treat these two cases.
	
	Our first main result concerns the same-mean case.
	We have not tried to optimize the constants in our results.
	
	\begin{theorem}[Total variation distance between Gaussians with the same mean]
		\label{thm:meanzero}
		Let $\mu\in \R^d$, $\Sigma_1$ and $\Sigma_2$ be positive definite $d\times d$ matrices, and  $\lambda_1,\dots,\lambda_d$ denote the eigenvalues of $\Sigma_1^{-1}\Sigma_2-I_d$.
		Then,
		\[
		\frac{1}{100}
		\leq
		\frac{\tv{\ncal(\mu,\Sigma_1)}{\ncal(\mu,\Sigma_2)}}{\min \left\{1, \sqrt{\displaystyle\sum_{i=1}^{d} \lambda_i^2} \right\}}
		\leq
		\frac32 .
		\]
		If $\Sigma_1$ and $\Sigma_2$ are positive semi-definite,
		$\range(\Sigma_1)=\range(\Sigma_2)$,
		and $r = \rank(\Sigma_1) = \rank(\Sigma_2)$, then let 
		$\Pi$ be a $d\times r$ matrix that has the same range as $\Sigma_1$ and $\Sigma_2$ and let $\rho_1,\dots,\rho_r$ denote the eigenvalues of $(\Pi\transpose \Sigma_1 \Pi)^{-1}(\Pi\transpose \Sigma_2 \Pi)-I_{r}$. Then, we have
		\[
	\frac{1}{100}
		\leq
		\frac{\tv{\ncal(\mu,\Sigma_1)}{\ncal(\mu,\Sigma_2)}}{\min \left\{1, \sqrt{ \displaystyle\sum_{i=1}^{r} \rho_i^2 }\right\}}
		\leq
		\frac32 .
		\]
	\end{theorem}

The paper~\cite{ulyanov} proves a bound similar to Theorem~\ref{thm:meanzero} for Gaussian distributions in a general Hilbert space: if $\Sigma_1$ and $\Sigma_2$ are positive definite matrices, $\Sigma_1^{-1}\Sigma_2-I$ has eigenvalues $\lambda_1,\cdots$, and 
$\sqrt{\sum \lambda_i^2} \leq 1/50$, then~\cite[Corollary~2]{ulyanov} gives
\[\frac{1}{100}
\leq
\frac{\tv{\ncal(\mu,\Sigma_1)}{\ncal(\mu,\Sigma_2)}}{ \sqrt{\sum \lambda_i^2}}
\leq
2 .\]
This result has the advantage that it covers infinite-dimensional spaces as well, but it holds only when 
${\sum \lambda_i^2}$ is smaller than a threshold.

One can express the quantities $\sum \lambda_i^2$ and $\sum \rho_i^2$ in Theorem~\ref{thm:meanzero} in terms of Frobenius norms of appropriate matrices.
For the first case, i.e., when $\Sigma_1,\Sigma_2$ are positive definite, we have
\begin{equation}
	\sum_{i=1}^{d} \lambda_i^2 = \tr \left( \left( \Sigma_1^{-1/2} \Sigma_2 \Sigma_1^{-1/2} - I_d \right)^2 \right) = \|\Sigma_1^{-1/2} \Sigma_2 \Sigma_1^{-1/2} - I_d \|_F^2.\label{trf}
\end{equation}
To see this, first note that 
$\Sigma_1^{-1/2} \Sigma_2 \Sigma_1^{-1/2}$
have the same spectrum as
$\Sigma_1^{-1} \Sigma_2$, because a vector $v$ is an eigenvector for $\Sigma_1^{-1} \Sigma_2$ with eigenvalue $\alpha$ if and only if $\Sigma_1^{1/2}v$ is an eigenvector for
$\Sigma_1^{-1/2} \Sigma_2 \Sigma_1^{-1/2}$
with eigenvalue $\alpha$.
Thus, the eigenvalues of $\left(\Sigma_1^{-1/2} \Sigma_2 \Sigma_1^{-1/2} - I_d\right)^2$ are 
$\lambda_1^2,\dots,\lambda_d^2$, proving the first equality in~\eqref{trf}.
The second equality follows by noting that the matrix $\Sigma_1^{-1/2} \Sigma_2 \Sigma_1^{-1/2} - I_d$ is symmetric.
The second case, i.e., when $\Sigma_1,\Sigma_2$ are positive semi-definite, can be handled similarly.



For the case where the means are different, we prove the following lower bound.
	
	\begin{theorem}[Total variation distance between  Gaussians with different means]\label{thm:main}
		Suppose $d>1$, let $\mu_1\neq\mu_2\in \R^d$ and let $\Sigma_1,\Sigma_2$ be positive definite $d\times d$ matrices.
		Let $v\coloneqq \mu_1-\mu_2$ and let $\Pi$ be a $d\times d-1$ matrix whose columns form a basis for the subspace orthogonal to $v$.
		Let $\rho_1,\dots,\rho_{d-1}$ denote the eigenvalues of $ (\Pi\transpose \Sigma_1 \Pi)^{-1}
		\Pi\transpose \Sigma_2 \Pi - I_{d-1}$.
		Define the function
		\[tv(\mu_1,\Sigma_1,\mu_2,\Sigma_2) \coloneqq 
		\max\left\{ \frac{|v\transpose (\Sigma_1-\Sigma_2)v|}{v\transpose \Sigma_1 v},
		{\frac{{v\transpose v}}{\sqrt{v\transpose \Sigma_1 v}}},\sqrt{\displaystyle\sum_{i=1}^{d-1} \rho_i^2}
		\right\}
		.\]
		Then, we have
		\[
		\frac{\min\{1,tv(\mu_1,\Sigma_1,\mu_2,\Sigma_2)\}}{200} \leq
		{\tvn{\mu_1}{\Sigma_1}{\mu_2}{\Sigma_2}}.
		\]
	\end{theorem}
	
	Note that the positive definiteness of the covariance matrices can be assumed without loss of generality: if $\mu_1+\range(\Sigma_1)=\mu_2+\range(\Sigma_1)\neq \R^d$, then one can work in this affine subspace instead.
	
	Along the way of proving this theorem, we also give  bounds for the one-dimensional case.

	\begin{theorem}[Total variation distance between one-dimensional Gaussians]
		\label{thm:onedimensional}
		In the one-dimensional case, $d=1$, we have
		\[
		\frac{1}{200} \min \left\{1, \max \left\{ \frac{|\sigma_1^2-\sigma_2^2|}{\sigma_1^2} , \frac{40 |\mu_1-\mu_2|}{\sigma_1} \right\} \right\}
		\leq
		\tv{\ncal(\mu_1,\sigma_1^2)}{\ncal(\mu_2,\sigma_2^2)}
		\leq \frac{3|\sigma_1^2-\sigma_2^2|}{2\sigma_1^2}+ \frac{|\mu_1-\mu_2|}{2\sigma_1} .
		\]
	\end{theorem}

Although the total variation distance is symmetric, our lower and upper bounds are not symmetric, so they can be automatically strengthened; for instance, the following symmetric version of Theorem~\ref{thm:onedimensional} holds:
		\begin{align*}
\frac{1}{200} \min \left\{1, \max \left\{ \frac{|\sigma_1^2-\sigma_2^2|}{\min\{\sigma_1,\sigma_2\}^2} , \frac{40 |\mu_1-\mu_2|}{\min\{\sigma_1,\sigma_2\}} \right\} \right\}
& \leq
\tv{\ncal(\mu_1,\sigma_1^2)}{\ncal(\mu_2,\sigma_2^2)}
\\& \leq \frac{3|\sigma_1^2-\sigma_2^2|}{2\max\{\sigma_1,\sigma_2\}^2}+ \frac{|\mu_1-\mu_2|}{2\max\{\sigma_1,\sigma_2\}} .
\end{align*}
Moreover, for Theorem~\ref{thm:meanzero}, swapping $\Sigma_1$ and $\Sigma_2$ can change the estimation of the total variation distance by at most a multiplicative factor of 2. Namely,
suppose $\Sigma_1$ and $\Sigma_2$ are positive definite $d\times d$ matrices, $\lambda_1,\dots,\lambda_d$ are the eigenvalues of $\Sigma_1^{-1}\Sigma_2-I_d$,
and
$\nu_1,\dots,\nu_d$ are the eigenvalues of $\Sigma_2^{-1}\Sigma_1-I_d$.
Then elementary calculations give
\[
\frac 1 2 \leq 
\frac
{\min \left\{1, \sqrt{\sum \lambda_i^2} \right\}}
{\min \left\{1, \sqrt{\sum \nu_i^2} \right\}}
\leq 2.
\]
	
	Some preliminaries and other known bounds for the total variation distance between Gaussians appear in Section~\ref{sec:prelim}.
	We start by proving
	Theorem~\ref{thm:meanzero}
	in
	Section~\ref{sec:thm:meanzero},
	then
	we prove
	Theorem~\ref{thm:onedimensional}
	in
	Section~\ref{sec:thn:onedimensional}, and finally we prove Theorem~\ref{thm:main} 
	in Section~\ref{sec:thm:main}.


	\section{Preliminaries}
	\label{sec:prelim}
\paragraph{Matrix definitions.}
	The $d$-dimensional identity matrix is denoted $I_d$. The trace and determinant of a matrix $A$ are denoted $\tr(A)$ and $\det(A)$, respectively. The \emph{Frobenius norm} (also called the Hilbert–Schmidt norm or the Schur norm) of a matrix $A$ is denoted by $\|A\|_F \coloneqq \sqrt{\tr(AA\transpose)}$. Note that $\|A\|_F^2$ equals the sum of squares of entries of $A$. If $A$ is symmetric, $\|A\|_F^2$ equals the sum of squares of eigenvalues of $A$.
	For general background on matrix norms, see~\cite[Chapter~5]{Horn}.

	\paragraph{The coupling characterization of the total variation distance.}
	For two distributions $P$ and $Q$,
	a pair $(X,Y)$ of random variables defined on the same probability space is called a \emph{coupling} for $P$ and $Q$ if $X \sim P$ and $Y\sim Q$.
	An extremely useful property of the total variation distance is {\em the coupling characterization}: for any two distributions $P$ and $Q$,
	we have $\tv{P}{Q} \leq t$ if and only if there exists a coupling $(X,Y)$ for them such that $\p{X\neq Y} \leq t$
	(see, e.g., \cite[Proposition~4.7]{Levin}).
	This characterization implies that for any function $f$ we have 
	$\tv{f(X)}{f(Y)} \leq \tv{X}{Y}$.
	If $f$ is invertible (for instance if $f(v)=Av+b$ where $A$ is full-rank) this also implies 
	$\tv{f(X)}{f(Y)} = \tv{X}{Y}$.
	
	An important property of the Gaussian distribution is that any linear transformation of a Gaussian random variable is also Gaussian: if $X\sim \ncal(\mu,\Sigma)$ then
	$$AX + b \sim \ncal (A\mu +b, A\Sigma A + A\mu b\transpose + b\mu\transpose A\transpose + b b\transpose).$$
	
	For a positive semi-definite matrix $\Sigma$ with eigendecomposition $\Sigma = \sum_{i = 1}^d \lambda_i v_i v_i\transpose$ where the $v_i$ are orthonormal, we define 
	$\Sigma^{1/2} \coloneqq \sum_{i = 1}^d \sqrt{\lambda_i} v_i v_i\transpose$ and $\Sigma^{-1/2} \coloneqq \sum_{i = 1}^d  v_i v_i\transpose/\sqrt{\lambda_i}$.
	It is easy to observe that if $g\sim \ncal(0,I)$ then $\Sigma^{1/2} g \sim \ncal(0,\Sigma)$.
	
	We will use the inequality
		\[
		0 \leq x-\log(1+x) \leq x^2
		\qquad \forall x \geq -2/3
		\]
	throughout, which implies that for any $x\geq -2/3$ there exists a $b\in[0,1]$ such that $x-\log(1+x)=bx^2$.
	
	We next state some known bounds for the total variation distance between two Gaussians,
	which may be more convenient than the above bounds for some applications. 

	For the case when the two Gaussians have the same covariance matrix, \cite[Theorem~1]{ulyanov} gives
\[
\tvn{\mu_1}{\Sigma}{\mu_2}{\Sigma}
= \p{ N(0,1) \in \left[
-\frac{ \sqrt{(\mu_1-\mu_2)\transpose \Sigma^{-1} (\mu_1-\mu_2)} }{2},
\frac{\sqrt{(\mu_1-\mu_2)\transpose \Sigma^{-1} (\mu_1-\mu_2)}}{2}
\right]}.
\]

The following bounds follow from known relations between statistical distances.
	
	\paragraph{An upper bound for the total variation distance using the KL-divergence.}
	For distributions $P$ and $Q$ over $\R^d$ with densities $p$ and $q$, their Kullback–Leibler divergence (KL-divergence) is defined as 
	\begin{align*}
		\kl{P}{Q} \coloneqq \int_{\R^d} p(x) \log\left( \frac{p(x)}{q(x)} \right) dx,
	\end{align*}
	and Pinsker's inequality~\cite[Lemma~2.5]{Tsybakov} states that $\tv{P}{Q}\leq \sqrt{\kl{P}{Q}/2}$ for any pair of distributions.
	The KL-divergence between two Gaussians has a closed form (e.g., \cite[Formula~(A.23)]{Rasmussen}):
	\[
	\operatorname{KL}(\ncal(\mu_1,\Sigma_1)\parallel{\ncal(\mu_2,\Sigma_2)})
	~=~
	\frac{1}{2} \left( \tr(\Sigma_1^{-1}\Sigma_2 - I)
	+ (\mu_1-\mu_2) \transpose \Sigma_1^{-1} (\mu_1-\mu_2)
	- \log \det (\Sigma_2 \Sigma_1^{-1}) \right).
	\]
	Combining these gives the following proposition.
	\begin{proposition}
		\label{upperkl}
		If $\Sigma_1$ and $\Sigma_2$ are positive definite, then
		\[\tv{\ncal(\mu_1,\Sigma_1)}{\ncal(\mu_2,\Sigma_2)} \leq \frac 1 2 \sqrt { \tr(\Sigma_1^{-1}\Sigma_2 - I)
			+ (\mu_1-\mu_2) \transpose \Sigma_1^{-1} (\mu_1-\mu_2)
			- \log \det (\Sigma_2 \Sigma_1^{-1}) }.\]	
	\end{proposition}
	
	\paragraph{Bounds for the total variation distance using the Hellinger distance.}
	For distributions $P$ and $Q$ over $\R^d$ with densities $p$ and $q$, their Hellinger distance is defined as 
	\begin{align*}
		\operatorname{H}({P},{Q}) \coloneqq \frac{1}{\sqrt2} \sqrt{ \int_{\R^d}  \left(\sqrt {p(x)}-\sqrt{q(x)}\right)^2 \, dx  } ,
	\end{align*}
	and it is known that $$\operatorname{H}({P},{Q})^2 \leq \tv{P}{Q} \leq 
	\operatorname{H}({P},{Q}) \sqrt{2-\operatorname{H}({P},{Q})^2}
	\leq\sqrt2 \operatorname{H}({P},{Q}),
	$$ 
	see~\cite[page~44]{cam2000asymptotics}. 
	The Hellinger distance between two Gaussians has a closed form (e.g., \cite[Exercises~11 and~14 in Chapter~1]{Pardo}):
	\[
	\operatorname{H}(\ncal(\mu_1,\Sigma_1), {\ncal(\mu_2,\Sigma_2)})^2
	~=~
	1 - \frac{\det(\Sigma_1)^{1/4}\det(\Sigma_2)^{1/4}}{\det\left(\frac{\Sigma_1+\Sigma_2}{2}\right)^{1/2}} \exp \left\{ -\frac18 (\mu_1-\mu_2)\transpose \left(\frac{\Sigma_1+\Sigma_2}{2}\right)^{-1}(\mu_1-\mu_2) \right\}.
	\]
	Combining these gives the following proposition.
	\begin{proposition}\label{upperhellinger} 
		Assume that $\Sigma_1,\Sigma_2$ are positive definite, and let $$h=h(\mu_1,\Sigma_1,\mu_2,\Sigma_2) \coloneqq \left(1 - \frac{\det(\Sigma_1)^{1/4}\det(\Sigma_2)^{1/4}}{\det\left(\frac{\Sigma_1+\Sigma_2}{2}\right)^{1/2}} \exp \left\{ -\frac18 (\mu_1-\mu_2)\transpose \left(\frac{\Sigma_1+\Sigma_2}{2}\right)^{-1}(\mu_1-\mu_2) \right\}\right)^{1/2}.$$
		Then, we have
		\[h^2 \leq \tv{\ncal(\mu_1,\Sigma_1)}{\ncal(\mu_2,\Sigma_2)} \leq h \sqrt{2-h^2} \leq h\sqrt2.\]
	\end{proposition}

	\section{The same-mean case: proof of Theorem~\ref{thm:meanzero}}
	\label{sec:thm:meanzero}
	In this section we consider the case when both Gaussians have the same mean.
	For proving the theorem we will need two lemmas.
	
	\begin{lemma}\label{lem:zeromeancore}
		Suppose $\lambda_1,\dots,\lambda_d \geq -2/3$ and let
		$\rho \coloneqq \sqrt{\sum_{i=1}^{d} \lambda_i^2}$. 
		If $C$ is a diagonal matrix with diagonal entries $1+\lambda_1,\dots,1+\lambda_d$, then
		\(
		\tvn{0}{C^{-1}}{0}{I_d} 
		\geq
		\rho/6 - \rho^2/8 - (e^{\rho^2}-1)/2.
		\)
	\end{lemma}
	
	\begin{proof} Define a random vector $g=(g_1,\dots,g_d)\sim\ncal(0,I_d)$. From~\eqref{tvl1} we have
		\begin{align*}
			2 \tv{\ncal(0,C^{-1})}{\ncal(0,I)} &= (2\pi)^{-d/2}\int_{\R^d} \left| e^{-x\transpose x/2} - \sqrt{\det(C)} e^{-x\transpose C x/2} \right| \, dx \\
			& = (2\pi)^{-d/2}\int_{\R^d} e^{-x\transpose x/2} \left| 1 - \sqrt{\det(C)} e^{-x\transpose (C-I_d) x/2} \right| \, dx \\
			& = \E \left| 1 - \sqrt{\det(C)} e^{-g\transpose (C-I_d) g/2} \right| \\
			& = \E \left| 1 - \exp \left(
			\sum_{i=1}^{d} \log(1+\lambda_i)/2 - \lambda_i g_i^2/2
			\right) \right|.
		\end{align*}
		Since $\lambda_i\geq -2/3$ for all $i$, we have
		$\log (1+\lambda_i)/2
		= \lambda_i/2 - b_i \lambda_i^2/2$
		for some $b_i\in[0,1]$, 
		and summing these up we find
		$\sum_{i = 1}^d \log (1+\lambda_i)/2
		= \sum_{i = 1}^d \lambda_i/2 - b \rho^2$
		for some $b\in[0,1]$.
		Also let $h_i = 1 - g_i^2$
		and $X = \sum_{i = 1}^d \lambda_i h_i/2$, whence
		\begin{align}
			2 \tv{\ncal(0,C^{-1})}{\ncal(0,I_d)} &=\E \left| 1 - e^{-b\rho^2} e^{X}  \right|
			\notag\\&\geq \E \left| 1 - e^X \right| - \E \left| e^X - e^{-b\rho^2} e^{X}  \right| \notag\\
			& \geq 
			\E \left| X \right| - \E X^2/2 - (1-e^{-b\rho^2})\E e^{X} \notag\\ &
			\geq 
			\frac {(\E X^2)^{3/2}}{(\E X^4)^{1/2}} - \E X^2/2 - (1-e^{-b\rho^2})\E e^{X}
	\label{3piece}	\end{align}
		where the first inequality is the triangle inequality, the second one follows from 
		\[
		|1 - e^x| \geq |x| - x^2/2 \qquad \forall x \in \R,
		\]
		and the third one follows from H\"older's inequality (see, e.g.,  \cite[Lemma~14.8]{florescu2013handbook}).
		We control each term on the right-hand-side of~\eqref{3piece}. First, observe that since $h_i$ is mean-zero, we have $\E h_i h_j=0$ for all $i\neq j$, and since $\E h_i^2 = 2$,
		\[
		\E X^2 = \E \left(\sum_{i = 1}^d \lambda_i h_i/2\right)^2
		= \sum_{i = 1}^d (\lambda_i/2)^2 \E h_i^2
		= \sum_{i = 1}^d \lambda_i^2/2 = \rho^2/2. 
		\]
		Second, since $\E g_i^2 = 1,\E g_i^4 = 3,\E g_i^6 =15$, and $\E g_i^8 =105,$
		we have $\E h_i^4 = 60$; thus,
		\begin{align*}
			\E X^4 & = \E \left(\sum_{i = 1}^d \lambda_i h_i/2\right)^4 \\
			&= \sum_{i = 1}^d (\lambda_i/2)^4 \E h_i^4
			+ 3 \sum_{i\neq j} (\lambda_i/2)^2 (\lambda_j/2)^2 \E h_i^2\E h_j^2
			\\& = 60 \sum_{i = 1}^d (\lambda_i/2)^4 
			+ 12  \sum_{i\neq j} (\lambda_i/2)^2 (\lambda_j/2)^2 \\
			& \leq 60 \left(\sum_{i = 1}^d (\lambda_i/2)^2 \right)^2 = 15 \rho^4/4.
		\end{align*}
		Finally, for the exponential moment,  note that 
		$\E \exp(tg_i^2) = (1-2t)^{-1/2}$ for any $t<1/2$, hence
		\begin{align*}
			\E e^{X} = 
			\prod_{i = 1}^d \left(e^{\lambda_i/2}
			\E e^{-\lambda_i g_i^2 /2}\right)
			= \prod_{i = 1}^d \left(e^{\lambda_i/2}
			e^{\frac{-1}{2} \log (1+\lambda_i) }\right)
			=\exp\left( \sum_{i = 1}^d 
			\lambda_i/2
			-\log (1+\lambda_i)/2
			\right) = e^{b\rho^2},
		\end{align*} 
		consequently,
		\begin{align*}
			2 \tv{\ncal(0,C^{-1})}{\ncal(0,I_d)} 
			\geq
			\frac{(\rho^2/2)^{3/2}}{(15\rho^4/4)^{1/2}} - \rho^2/4
			- e^{b\rho^2}+1 \geq \rho/3 - \rho^2/4 - (e^{\rho^2}-1),
		\end{align*}
		completing the proof.
	\end{proof}
	
	\begin{lemma}\label{lem:onedimconstant}
		If $\lambda^2\geq 0.01$ then 
		$\tvn{0}{1}{0}{1+\lambda}>0.01$.
	\end{lemma}
	\begin{proof}
		If $\lambda>0$ then $1+\lambda\geq1.1$, so we have
		\begin{align*}
			\tvn{0}{1}{0}{1+\lambda} &\geq
			\p{N(0,1)\in[-1,1]}-\p{N(0,1+\lambda)\in[-1,1]} \\ &\geq
			\p{N(0,1)\in[-1,1]}-\p{N(0,1.1)\in[-1,1]} \\ &> 0.68-0.66 > 0.01,  
		\end{align*}
		and if $\lambda<0$ then $1+\lambda\leq0.9$, so we have
		\begin{align*}
			\tvn{0}{1}{0}{1+\lambda} &\geq
			\p{N(0,1+\lambda)\in[-1,1]}-\p{N(0,1)\in[-1,1]} \\ &\geq
			\p{N(0,0.9)\in[-1,1]}-\p{N(0,1)\in[-1,1]} \\ &> 0.70-0.69 = 0.01. \qedhere
		\end{align*}
	\end{proof}
	
	We can now prove
	Theorem~\ref{thm:meanzero}.
	
	\begin{proof}[Proof of Theorem~\ref{thm:meanzero}]
		For both parts of the theorem, we may assume that $\mu=0$.
		We start with the case that $\Sigma_1$ and $\Sigma_2$ are positive definite, i.e., they have full rank.
		Recall that $\Sigma_1^{-1}\Sigma_2$ have eigenvalues $1+\lambda_1,\dots,1+\lambda_d$.
		Let $\rho \coloneqq  \sqrt{\sum_{i = 1}^d {\lambda_i^2}}$.
		
		We first prove the upper bound.
		If some $\lambda_i<-2/3$ then trivially
		\[
		\tv{\ncal(0,\Sigma_1)}{\ncal(0,\Sigma_2)} \leq 1 \leq \frac32 |\lambda_i| \leq \frac32 \sqrt{\sum_{i=1}^{d} \lambda_i^2}=3\rho/2.
		\]
		Otherwise, by Proposition~\ref{upperkl},
		\[
		4\tv{\ncal(0,\Sigma_1)}{\ncal(0,\Sigma_2)}^2
		\leq  \sum_{i=1}^{d} (\lambda_i - \log(1+\lambda_i)) \leq  \sum_{i=1}^{d} \lambda_i^2=\rho^2, 
		\]
and the upper bound in the theorem is proved.
		
		For proving the lower bound, 
		we first claim that if $C$ is a diagonal matrix with diagonal entries $1+\lambda_1,\dots,1+\lambda_d$, then
		\begin{equation}
			\tvn{0}{\Sigma_1}{0}{\Sigma_2}
			=
			\tvn{0}{C^{-1}}{0}{I_d}.
			\label{diagonalize}
		\end{equation}
		To prove this, let $g\sim \ncal(0,I_d)$.
		We first claim if $E$ and $F$ are positive definite matrices with the same spectrum, then $\TV(Eg, g) = \TV(Fg, g)$. 
		To see this, let $s_1, \dots, s_d$ be the eigenvalues of $E$ and $F$, and let $g_1, \dots, g_d$ be the components of $g$. 
		By rotational invariance of the standard Gaussian distribution (see, e.g., \cite[Proposition~3.3.2]{Vershynin}), both  $\TV(Eg, g)$ and $\TV(Fg, g)$ are equal to 
		$\TV ( (s_1  g_1, s_2  g_2, \dots, s_d  g_d),  (g_1, g_2, \dots, g_d) )$, and the claim is proved.
		This also implies,
		for any two positive definite matrices $E$ and $F$ with the same spectrum, \[\tvn{0}{I_d}{0}{E}=\tvn{0}{I_d}{0}{F}.\]		
		Next, we have 
		\begin{align*}
        \TV ( \ncal(0, \Sigma_1), \ncal(0, \Sigma_2)) &= 
		\TV (\Sigma_1^{1/2} g, \Sigma_2^{1/2} g) = 
		\TV (\Sigma_2^{-1/2} \Sigma_1^{1/2} g,g) \\
		&= 
		\TV ( \ncal(0, \Sigma_2^{-1/2} \Sigma_1 \Sigma_2^{-1/2}),\ncal(0,I_d)).
		\end{align*}
		Now $\Sigma_2^{-1/2} \Sigma_1 \Sigma_2^{-1/2}$ 
		has the same spectrum
		as $\Sigma_2^{-1}\Sigma_1$, which has the same spectrum as $C^{-1}$,
		whence \eqref{diagonalize} is proved.
		
		For proving the lower bound in the theorem we consider three cases.
		
		\noindent\textbf{Case 1: there exists some $i$ with $|\lambda_i|\geq0.1$. }
		Observe that if we project a random variable distributed as $\ncal(0,C^{-1})$ onto the $i$th component, we obtain a one-dimensional $\ncal(0,(1+\lambda_i)^{-1})$ random variable .
		Since projection can only decrease the total variation distance, using Lemma~\ref{lem:onedimconstant} we obtain
		\[
		\tvn{0}{C^{-1}}{0}{I_d}
		\geq
		\tvn{0}{(1+\lambda_i)^{-1}}{0}{1}
		=
		\tvn{0}{1}{0}{1+\lambda_i}
		\geq 0.01,
		\]
		as required.
		The  above equality holds because the total variation distance is invariant under any linear transformation.
		
		\noindent\textbf{Case 2: $|\lambda_i|<0.1$ for all $i$, and $\rho \leq 0.17$.}
		In this case
		Lemma~\ref{lem:zeromeancore} gives
		\[
		\tvn{0}{C^{-1}}{0}{I_d}
		\geq
		\rho/6 - \rho^2/8 - (e^{\rho^2}-1)/2
		\geq \rho / 100,
		\]
		as required.
		
		\noindent\textbf{Case 3: $|\lambda_i|<0.1$ for all $i$, and $\rho > 0.17$.}
		Define 
		\[
		f(\rho)\coloneqq \rho/6 - \rho^2/8 - (e^{\rho^2}-1)/2,
		\]
		and observe that $f(x)\geq0.01$ for $0.1\leq x \leq 0.17$.
		Let $1\leq j < d$ be the largest index such that
		$\sum_{i=1}^{j} \lambda_i^2 \leq 0.17^2$,
		and observe that 
		since $|\lambda_i|<0.1$ for all $i$, we have
		$\rho'^2\coloneqq \sum_{i=1}^{j} \lambda_i^2 \geq 0.17^2 - 0.1^2 >0.01$ and so
		$f(\rho') \geq 0.01$.
		Let $C'$ be the diagonal $j\times j$ matrix with diagonal entries $1+\lambda_1,\dots,1+\lambda_j$.
		If we project a random variable distributed as $\ncal(0,C^{-1})$ onto the first $j$ coordinates, we obtain a $ \ncal(0,C'^{-1})$ random variable.
		Since projection can only decrease the total variation distance, using Lemma~\ref{lem:zeromeancore} we obtain
		\[
		\tvn{0}{C^{-1}}{0}{I_d}
		\geq
		\tvn{0}{C'^{-1}}{0}{I_j}
		\geq f(\rho') \geq 0.01,
		\]
		as required.
		
		We finally consider the second part of the theorem, i.e., when $\Sigma_1$ and $\Sigma_2$ are positive semi-definite and $\range(\Sigma_1)=\range(\Sigma_2)$.
		Recall that $\Pi$ is a $d\times r$ matrix whose columns form a basis for $\range(\Sigma_1)$. 
		Then observe that 
		$v \mapsto \Pi\transpose v$ is an invertible map from $\range(\Sigma_1)$ to $\R^{r}$, with the inverse given by $w \mapsto \Pi (\Pi\transpose \Pi)^{-1} w$.
		This implies
		\[
		\tv{N(0,\Sigma_1)}{N(0,\Sigma_2)}
		=
		\tv{\Pi\transpose N(0,\Sigma_1)}{\Pi\transpose N(0,\Sigma_2)}
		=
		\tv{\ncal(0,\Pi\transpose \Sigma_1 \Pi)}
		{\ncal(0,\Pi\transpose \Sigma_2 \Pi)}.
		\]
		The matrices $\Pi\transpose \Sigma_1 \Pi$ and $\Pi\transpose \Sigma_2 \Pi$ are  positive definite
		$r \times r$ matrices, hence the second part of the theorem follows from the first part.
	\end{proof}
	
	\section{The one-dimensional case: proof of Theorem~\ref{thm:onedimensional}}
	\label{sec:thn:onedimensional}
	We start with the upper bound. If
	$\frac{|\sigma_1^2-\sigma_2^2|}{\sigma_1^2}\geq 2/3$, then the right-hand-side is at least 1 and the bound holds because the total variation distance is at most 1. Otherwise, since ${\sigma_2^2}/{\sigma_1^2}-1\geq -2/3$, 
	we have ${\sigma_2^2}/{\sigma_1^2}-1-\log({\sigma_2^2}/{\sigma_1^2})\leq ({\sigma_2^2}/{\sigma_1^2}-1)^2$, so
	from Proposition~\ref{upperkl} we have
	\begin{align*}
		\tv{\ncal(\mu_1,\sigma_1^2)}{\ncal(\mu_2,\sigma_2^2)} & \leq \frac12 \sqrt{{\sigma_2^2}/{\sigma_1^2}-1-\log({\sigma_2^2}/{\sigma_1^2})  + (\mu_1-\mu_2)^2/\sigma_1^2} \\&\leq \frac12
		\sqrt{{\sigma_2^2}/{\sigma_1^2}-1-\log({\sigma_2^2}/{\sigma_1^2})} + \frac12\sqrt{(\mu_1-\mu_2)^2/\sigma_1^2} \\&\leq \frac12 |{\sigma_2^2}/{\sigma_1^2}-1| + \frac 12 |(\mu_1-\mu_2)/\sigma_1|,
	\end{align*}
	completing the proof of the upper bound.
	
	The lower bound follows from the following two lower bounds:
	\begin{align}
		\frac{1}{200} \min \left\{1,  \frac{|\sigma_1^2-\sigma_2^2|}{\sigma_1^2}\right\}
		& \leq
		\tv{\ncal(\mu_1,\sigma_1^2)}{\ncal(\mu_2,\sigma_2^2)}, \label{lb1} \\
		\frac15 \min \left\{1,   \frac{|\mu_1-\mu_2|}{\sigma_1} \right\}
		& \leq
		\tv{\ncal(\mu_1,\sigma_1^2)}{\ncal(\mu_2,\sigma_2^2)}. \label{lb2}
	\end{align}
	We start with proving~\eqref{lb1}.
	We show
	\begin{equation}
		\label{changemean}
		\frac12 
		\tv{\ncal(0,\sigma_1^2)}{\ncal(0,\sigma_2^2)} \leq 
		\tv{\ncal(\mu_1,\sigma_1^2)}{\ncal(\mu_2,\sigma_2^2)},
	\end{equation}
	and then \eqref{lb1} follows from Theorem~\ref{thm:meanzero}.
	Assume, without loss of generality, that $\sigma_1\leq\sigma_2$
	and $\mu_1\leq \mu_2$.
	By the form of the density of the normal distribution, this implies there exists some $c=c(\sigma_1,\sigma_2)$ such that
	$$\tv{\ncal(0,\sigma_1^2)}{\ncal(0,\sigma_2^2)}
	= \p{N(0,\sigma_2^2) \notin [-c,c]}- \p{N(0,\sigma_1^2) \notin [-c,c]},$$
	and thus
	$$
	\p{N(0,\sigma_2^2) >c}
	= \p{N(0,\sigma_1^2) >c} + \tv{\ncal(0,\sigma_1^2)}{\ncal(0,\sigma_2^2)}/2.$$
	Therefore,
	\begin{align*}
		\p{N(\mu_2,\sigma_2^2) >c}
		& = \p{N(\mu_2,\sigma_1^2) >c} + \tv{\ncal(0,\sigma_1^2)}{\ncal(0,\sigma_2^2)}/2
		\\ & \geq
		\p{N(\mu_1,\sigma_1^2) >c} + \tv{\ncal(0,\sigma_1^2)}{\ncal(0,\sigma_2^2)}/2,
	\end{align*}
	and \eqref{changemean} is proved.
	
	To complete the proof of Theorem~\ref{thm:onedimensional} we need only prove \eqref{lb2}.
	By symmetry, we may assume $\mu_1\leq \mu_2$.
	Let $X \sim \ncal(\mu_1,\sigma_1^2)$.
	Then
	\begin{align*}
		\tvn{\mu_1}{\sigma_1^2}{\mu_2}{\sigma_2^2} &\geq \p{N(\mu_2,\sigma_2^2)\geq \mu_2}
		- \p{X\geq \mu_2}\\
		& = 1/2 - (1/2-\p{X\in[\mu_1,\mu_2]})
		\\& =\p{X\in[\mu_1,\mu_2]}.
	\end{align*}
	If $\mu_2-\mu_1 \geq \sigma_1$, then
	\[
	\p{X\in[\mu_1,\mu_2]}
	\geq
	\p{X\in[\mu_1,\mu_1+\sigma_1]}
	=
	\p{N(0,1)\in[0,1]} > \frac15,
	\]
	and if $\mu_2-\mu_1 < \sigma_1$ then
	\begin{align*}
		\p{X\in[\mu_1,\mu_2]}
		= \int_{\mu_1}^{\mu_2} \frac{e^{-(x-\mu_1)^2/2\sigma_1^2}}{\sqrt{2\pi}\sigma_1} dx \geq (\mu_2-\mu_1) \frac{e^{-(\mu_2-\mu_1)^2/2\sigma_1^2}}{\sqrt{2\pi}\sigma_1} > \frac{e^{-1/2}}{\sqrt{2\pi}} \frac{|\mu_1-\mu_2|}{\sigma_1} >   \frac{ |\mu_1-\mu_2|}{5\sigma_1},
	\end{align*}
	which proves~\eqref{lb2} and completes the proof of the theorem.

	\section{The general case: proof of Theorem~\ref{thm:main}}\label{sec:thm:main}
	Recall that $v=\mu_1-\mu_2$, and let $u\coloneqq (\mu_1+\mu_2)/2$. Any vector in $\R^d$ has a component in the direction of $v$ and a component orthogonal to $v$. In particular, any $w$ 
	can be written uniquely as \[w = u + f_1(w) v + f_2(w),\qquad f_2(w)\transpose v = 0,\]
	with $f_1$ and $f_2$ given by
	$$f_1(w) = \frac{(w-u)\transpose v}{v\transpose v}\in \R, \qquad f_2(w) = w-u-f_1(w)v = P(w-u),
	$$
	with $P\coloneqq I_d - vv\transpose/v\transpose v$.
	
	Let $X\sim \ncal(\mu_1,\Sigma_1)$
	and $Y\sim \ncal(\mu_2,\Sigma_2)$.
	Then we have, by the coupling characterization of the total variation distance,
	\begin{align*}
		\max \{ \tv{f_1(X)}{f_1(Y)}, \tv{f_2(X)}{f_2(Y)} \}
	    \leq \tv{X}{Y}.
	\end{align*}
		We next claim that $\displaystyle f_1(X)\sim \ncal\left(\frac 1 2 ,\frac{v\transpose \Sigma_1 v}{\red{(v\transpose v)^2}} \right)$. To see this, observe that $f_1(X)$ is a linear map of a Gaussian, so it is Gaussian. Its mean and covariance can be computed from those of $X$. Similarly, one can compute $\displaystyle f_1(Y)\sim \ncal\left(-\frac 1 2,\frac{v\transpose \Sigma_2 v}{\red{(v\transpose v)^2}}\right)$. So, Theorem~\ref{thm:onedimensional} gives
	\[
	\frac{1}{200} \min \left\{1, \max \left\{ \frac{|v\transpose \Sigma_1 v-v\transpose \Sigma_2 v|}{v\transpose \Sigma_1 v} , 40\red{\frac{{v\transpose v} }{{\sqrt{v\transpose \Sigma_1 v}}}} \right\} \right\}
	\leq
	\tv{f_1(X)}{f_1(Y)} .
	\]
	
	On the other hand, since  $f_2(w)=P(w-u)$, both $f_2(X)$ and $f_2(Y)$ are also Gaussians, with $f_2(X)\sim\ncal(0,P\Sigma_1P)$ and
	$f_2(Y)\sim\ncal(0,P\Sigma_2P)$. Note that 
	$\range(P\Sigma_1P)=\range(P\Sigma_2P)=\range(\Pi)$.
	Also observe that since each column of $\Pi$ is orthogonal to $v$, we have $\Pi\transpose P = \Pi$ and $P \Pi = \Pi$.
	Recall that  $\rho_1,\dots,\rho_{d-1}$ are the eigenvalues of $ (\Pi\transpose \Sigma_1 \Pi)^{-1}
	\Pi\transpose \Sigma_2 \Pi - I_{d-1}$.
	Hence the second part of Theorem~\ref{thm:meanzero} gives
	\[
	\frac{1}{100} \min \left\{1, \sqrt{\sum_{i=1}^{d-1}\rho_i^2}\right\}
	\leq
	\tv{f_2(X)}{f_2(Y)},
	\]
	completing the proof of Theorem~\ref{thm:main}.

\paragraph{Acknowledgments.} We are grateful to
Michael Kohler, Gautam Kamath, Cole Franks,
and Shirshendu Ganguly for pointing out inaccuracies in earlier versions of this paper.
		
	\bibliographystyle{plain}
	\bibliography{tvdistance}

\begin{thebibliography}{10}

\bibitem{tight_bounds}
Jamil Arbas, Hassan Ashtiani, and Christopher Liaw.
\newblock Polynomial time and private learning of unbounded {G}aussian mixture
  models.
\newblock {\em arXiv preprint arXiv:2303.04288v2 [stat.ML]}, 2023.
\newblock Published in Proceedings of the 40th International Conference on
  Machine Learning.

\bibitem{ulyanov}
S.~S. {Barsov} and V.~V. {Ul'yanov}.
\newblock {Estimates of the proximity of Gaussian measures.}
\newblock {\em {Sov. Math., Dokl.}}, 34:462--466, 1987.

\bibitem{florescu2013handbook}
Ionut Florescu and Ciprian~A Tudor.
\newblock {\em Handbook of Probability}.
\newblock John Wiley \& Sons, 2013.

\bibitem{Horn}
Roger~A. Horn and Charles~R. Johnson.
\newblock {\em Matrix Analysis}.
\newblock Cambridge University Press, Cambridge, second edition, 2013.

\bibitem{cam2000asymptotics}
Lucien Le~Cam and Grace~Lo Yang.
\newblock {\em Asymptotics in Statistics: Some Basic Concepts}.
\newblock Springer Series in Statistics. Springer New York, second edition,
  2000.

\bibitem{Levin}
David~A. Levin and Yuval Peres.
\newblock {\em Markov Chains and Mixing Times, second edition, with
  contributions by Elizabeth L. Wilmer, with a chapter on ``coupling from the
  past'' by James G. Propp and David B. Wilson}.
\newblock American Mathematical Society, Providence, RI, 2017.

\bibitem{Pardo}
Leandro Pardo.
\newblock {\em Statistical Inference Based on Divergence Measures}, volume 185
  of {\em Statistics: Textbooks and Monographs}.
\newblock Chapman \& Hall/CRC, Boca Raton, FL, 2006.

\bibitem{Rasmussen}
Carl~Edward Rasmussen and Christopher K.~I. Williams.
\newblock {\em Gaussian Processes for Machine Learning}.
\newblock Adaptive Computation and Machine Learning. MIT Press, Cambridge, MA,
  2006.

\bibitem{Verdu}
Igal Sason and Sergio Verd\'{u}.
\newblock {$f$}-divergence inequalities.
\newblock {\em IEEE Trans. Inform. Theory}, 62(11):5973--6006, 2016.

\bibitem{Tong}
Y.~L. Tong.
\newblock {\em The Multivariate Normal Distribution}.
\newblock Springer Series in Statistics. Springer-Verlag, New York, 1990.

\bibitem{Tsybakov}
Alexandre~B. Tsybakov.
\newblock {\em Introduction to Nonparametric Estimation}.
\newblock Springer Series in Statistics. Springer, New York, 2009.
\newblock Revised and extended from the 2004 French original, Translated by
  Vladimir Zaiats.

\bibitem{Vershynin}
Roman Vershynin.
\newblock {\em High-Dimensional Probability: An Introduction with Applications
  in Data Science}.
\newblock Cambridge Series in Statistical and Probabilistic Mathematics.
  Cambridge University Press, 2018.

\end{thebibliography}
	
\end{document}